\documentclass[11pt,bezier]{article}
\usepackage{amsmath, amssymb, amsfonts, graphicx}
\usepackage{pgf,tikz}\usepackage{mathrsfs}

\newtheorem{prethm}{{\bf Theorem}}

\newenvironment{thm}{\begin{prethm}\sl{\hspace{-0.5
               em}{\bf.}}}{\end{prethm}}

\newtheorem{prepro}[prethm]{{\bf Proposition}}

\newtheorem{prelem}[prethm]{{\bf Lemma}}

\newenvironment{lem}{\begin{prelem}\sl{\hspace{-0.5
               em}{\bf.}}}{\end{prelem}}

\newtheorem{predeff}[prethm]{{\bf Definition}}

\newtheorem{precor}[prethm]{{\bf Corollary}}

\newtheorem{preconj}[prethm]{{\bf Conjecture}}

\newtheorem{preremark}[prethm]{{\bf Remark}}

\newenvironment{remark}{\begin{preremark}\rm{\hspace{-0.5
               em}{\bf.}}}{\end{preremark}}

\newtheorem{preexample}[prethm]{{\bf Example}}

\newtheorem{preproof}{{\bf\textsf{Proof.}}}

\newenvironment{proof}[1]{\begin{preproof}{\rm
               #1}\hfill{$\Box$}}{\end{preproof}}

\newcommand{\la}{\lambda}

\newcommand{\mul}{{\rm mult}}

\newcommand{\nsg}{{\rm NSG}}

\title{Eigenvalue-free interval for threshold graphs}

\author{{\sc Ebrahim Ghorbani} \\[.3cm]
{\sl Department of Mathematics, K. N. Toosi University of Technology,}\\
{\sl P. O. Box 16765-3381, Tehran, Iran}\\
{\sl School of Mathematics, Institute for Research in Fundamental
Sciences (IPM),}\\
{\sl P. O. Box 19395-5746, Tehran, Iran}
\\[.3cm]
$\mathsf{e\_ghorbani@ipm.ir}$}

\begin{document}
\maketitle

\vspace{5mm}

\begin{abstract}
This paper deals with the eigenvalues of the adjacency matrices of threshold graphs for which $-1$ and $0$ are considered as trivial eigenvalues.
We show that threshold  graphs have no non-trivial eigenvalues in the interval $\left[(-1-\sqrt{2})/2,\,(-1+\sqrt{2})/2\right]$.
This confirms a conjecture by Aguilar,  Lee,  Piato, and  Schweitzer (2018).

\vspace{5mm}
\noindent {\bf AMS Mathematics Subject Classification\,(2010):}   05C50, 05C75
\end{abstract}

\vspace{5mm}

\section{Introduction}
A threshold graph is a graph that can be constructed from a one-vertex graph by repeated addition of a single isolated vertex to the graph, or
addition of a single vertex that is adjacent to all other vertices.
An equivalent definition is the following: a graph is a threshold graph if there are a real number $S$ and for each vertex $v$ a real weight $w(v)$ such that two vertices $u,v$ are adjacent  if and only if $w(u)+w(v)> S$.
This justifies the name ``threshold graph" as $S$ is the threshold for  being adjacent.
Threshold graphs also can  be defined in terms of forbidden subgraphs, namely they are  $\{P_4,2K_2,C_4\}$-free
graphs.  Note, if a threshold graph is
not connected then (since $2K_2$ is forbidden) at most one of its
components is non-trivial (others are trivial, i.e. isolated
vertices). For more information on properties of threshold graphs and related family of graphs see \cite{bls,mp}.

In this paper we deal with  eigenvalues of the  adjacency matrices of threshold graphs for which $-1$ and $0$ are considered as trivial eigenvalues.
In \cite{jtt0} it was shown that threshold graphs have no eigenvalues in $(-1, 0)$.
This result was extended in \cite{gh} by showing that a graph $G$ is a cograph  (i.e. a $P_4$-free graph) if and only if no induced subgraph of $G$ has an eigenvalue in the interval $(-1,0)$. A distinguished subclass of threshold graphs is the family of {\em anti-regular graphs} which are the graphs with only two vertices of equal degrees.  If $G$ is anti-regular it follows easily that the complement graph $\overline{G}$ is also anti-regular.  Up to isomorphism, there is only one connected anti-regular graph on $n$ vertices  and  its complement is the unique disconnected $n$-vertex anti-regular graph \cite{bc}.

In this paper we show that  threshold  graphs have no non-trivial eigenvalues in the interval $\left[(-1-\sqrt{2})/2,\,(-1+\sqrt{2})/2\right]$.
This result confirms a conjecture of \cite{alps} and improve the aforementioned result of \cite{jtt0}.
We remark that another related conjecture was posed in \cite{alps}  that for any $n$, the anti-regular graph with $n$ vertices has the smallest positive eigenvalue
and has the largest non-trivial negative eigenvalue among all threshold graphs on $n$ vertices.
Partial results on this conjecture is given in \cite{afss}.

\section{Preliminaries}\label{pre}

In this section we introduce the notations and recall  a basic fact  which will be used in the sequel.
The graphs we consider are all simple and undirected.
For a  graph $G$, we denote  by $V(G)$ the vertex set of $G$.
For two vertices $u,v$, by $u\sim v$ we mean $u$ and $v$ are adjacent.
 If  $V(G)=\{v_1, \ldots , v_n\}$, then the {\em adjacency matrix} of $G$ is an $n \times  n$
 matrix $A(G)$ whose $(i, j)$-entry is $1$ if $v_i\sim v_j$ and  $0$ otherwise.
 By {\em eigenvalues} of $G$ we mean those of $A(G)$.
 The multiplicity of an eigenvalue $\la$ of $G$ is denoted by $\mul(\la,G)$.
For a vertex $v$ of $G$, let $N_G(v)$ denote the {\em open neighborhood} of $v$, i.e.   the set of
vertices of $G$ adjacent to $v$ and $N_G[v]=N_G(v)\cup\{v\}$ denote the {\em closed neighborhood} of $v$; we will drop
the subscript $G$  when it is clear from the context.
Two vertices $u$ and $v$ of $G$ are called {\em duplicates} if $N(u)=N(v)$
and called {\em coduplicates} if $N[u]=N[v]$. Note that duplicate vertices cannot be
adjacent while coduplicate vertices must be adjacent.
A subset $S$ of  $V(G)$   such that $N(u)=N(v)$
for any $u, v\in  S$   is called  a  {\em duplication  class} of $G$.  Coduplication  classes  are defined analogously.

We will make use of the interlacing property of graph eigenvalues which we recall below (see \cite[Theorem~2.5.1]{bh}).
\begin{lem}\label{inter}
Let $G$ be a graph of order $n$, $H$ be an induced subgraph of $G$ of order $m$, $\la_1\ge\cdots\ge\la_n$ and $\mu_1\ge\cdots\ge\mu_m$ be the eigenvalues of $G$ and $H$, respectively. Then $$\la_i\ge\mu_i\ge\la_{n-m+i}~~\hbox{for}~ i=1, \ldots,m.$$
In particular, if $m=n-1$, then
 $$\la_1\ge\mu_1\ge\la_2\ge\mu_2\ge\cdots\ge\la_{n-1}\ge\mu_{n-1}\ge\la_n.$$
\end{lem}

\section{Eigenvalue-free interval for threshold graphs}

In this section we present the main result of the paper. We start by the following remark on the structure of threshold graphs.

\begin{remark}\label{thrstruc}  As it was observed in \cite{ma} (see also \cite{as,ha}),
the vertices of any connected threshold graph $G$ can be partitioned into $h$ non-empty coduplication classes $V_1,\ldots, V_h$ and $h$ non-empty duplication classes $U_1,\ldots, U_h$  such that the vertices in $V_1\cup\cdots\cup V_h$ form a clique and
$$N(u)=V_1\cup\cdots\cup V_i~~\hbox{for any}~ u\in U_i,~1\le i\le h.$$
(It turns out that  $U_1\cup\cdots\cup U_h$ form a coclique.)
Accordingly, a connected threshold graph is also called {\em nested
split graph} (or NSG for short). If $m_i = |U_i|$ and $n_i =|V_i|$ for $1\le i\le h$,
then we write
$$G =\nsg(m_1,\ldots , m_h; n_1, \ldots, n_h).$$
For an illustration of this structure with $h=5$, see Figure~\ref{figthr}.
It follows that a threshold graph $G$ of order $n$ is anti-regular if and only if $n_1=\cdots=n_h=m_1=\cdots=m_h=1$ (in case $n$ is even) or
$n_1=\cdots=n_h=m_1=\cdots=m_{h-1}=1$ and $m_h=2$ (in case $n$ is odd).
\end{remark}

\begin{figure}\centering
\begin{tikzpicture}
\draw [line width=1.75pt] (-4,-3)-- (-4,-7);
\draw [line width=1.75pt] (-4,-3)-- (-3.4,-4.19);
\draw [line width=1.75pt] (-3.4,-4.19)-- (-3.38,-5.9);
\draw [line width=1.75pt] (-3.38,-5.9)-- (-4,-7);
\draw [line width=1.75pt] (-3.4,-4.19)-- (-4,-7);
\draw [line width=1.75pt] (-4,-3)-- (-4.28,-4.91);
\draw [line width=1.75pt] (-4,-3)-- (-3.38,-5.9);
\draw [line width=1.75pt] (-4.28,-4.91)-- (-4,-7);
\draw [line width=1.75pt] (-4.28,-4.91)-- (-3.4,-4.19);
\draw [line width=1.75pt] (-4.28,-4.91)-- (-3.38,-5.9);
\draw [line width=1.75pt] (-7,-3)-- (-4,-3);
\draw [line width=1.75pt] (-7,-4)-- (-4,-3);
\draw [line width=1.75pt] (-7,-4)-- (-3.4,-4.19);
\draw [line width=1.75pt] (-7,-5)-- (-4,-3);
\draw [line width=1.75pt] (-7,-5)-- (-3.4,-4.19);
\draw [line width=1.75pt] (-7,-5)-- (-4.28,-4.91);
\draw [line width=1.75pt] (-7,-6)-- (-4,-3);
\draw [line width=1.75pt] (-7,-6) to [out=35,in=-160] (-3.4,-4.19);
\draw [line width=1.75pt] (-7,-6)-- (-4.28,-4.91);
\draw [line width=1.75pt] (-7,-6)-- (-3.38,-5.9);
\draw [line width=1.75pt] (-7,-7)-- (-4,-3);
\draw [line width=1.75pt] (-7,-7)-- (-4.28,-4.91);
\draw [line width=1.75pt] (-7,-7)-- (-3.38,-5.9);
\draw [line width=1.75pt] (-7,-7)-- (-4,-7);
\draw[line width=1.75pt] (-7,-7) to [out=30,in=-120] (-3.4,-4.19);
\draw [fill=black] (-4,-3) circle (5.5pt);
\draw[color=black] (-3.6,-2.8) node {$V_1$};
\draw [fill=black] (-3.4,-4.19) circle (5.5pt);
\draw[color=black] (-3.2,-3.8) node {$V_2$};
\draw [fill=black] (-4.28,-4.91) circle (5.5pt);
\draw[color=black] (-4.5,-5.3) node {$V_3$};
\draw [fill=black] (-3.38,-5.9) circle (5.5pt);
\draw[color=black] (-2.9,-6.1) node {$V_4$};
\draw [fill=black] (-4,-7) circle (5.5pt);
\draw[color=black] (-3.5,-7.1) node {$V_5$};
\draw(-7,-4) circle (5.5pt);
\draw(-7,-3) circle (5.5pt);
\draw (-7,-5) circle (5.5pt);
\draw (-7,-6) circle (5.5pt);
\draw (-7,-7) circle (5.5pt);
\draw (-7.5,-3) node {$U_1$};
\draw (-7.5,-4) node {$U_2$};
\draw (-7.5,-5) node {$U_3$};
\draw (-7.5,-6) node {$U_4$};
\draw (-7.5,-7) node {$U_5$};
\end{tikzpicture}
\caption{A threshold graph: $V_i$'s are cliques, $U_i$'s are cocliques, each thick line indicates the edge set of a complete bipartite subgraph on some $U_i,V_j$ }\label{figthr}
\end{figure}
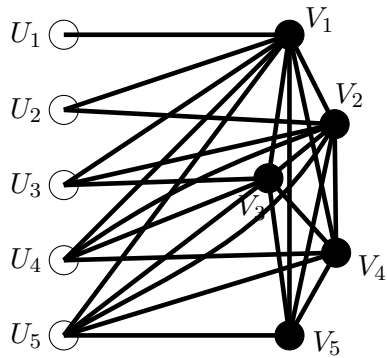

In any graph $G$  if we add a new vertex duplicate (coduplicate) to $u\in V(G)$, then the multiplicity of $0$ (of $-1$) increases by 1.
That's why the eigenvalues 0 and $-1$ are treated as trivial eigenvalues in threshold graphs. The following lemma can be deduced in a similar manner.

\begin{lem}\label{threshmult} Let $G=\nsg(m_1,\ldots,m_h;n_1,\ldots,n_h)$ be a connected threshold graphs. Then
\begin{align*}
\mul(0,G)&=\sum_{i=1}^h(m_i-1),\\
\mul(-1,G)&=\sum_{i=1}^h(n_i-1)+\left\{\begin{array}{ll}1&\hbox{if $m_h=1$},\\0&\hbox{if $m_h\ge2$.}\end{array}\right.
\end{align*}
\end{lem}

\begin{lem}\label{G-v} Let $G$ be threshold graph which is not an anti-regular graph. Then there is some vertex $v$ of $G$ such that for $H=G-v$ we have either
\begin{itemize}
  \item[\rm(i)] $\mul(0,G)=\mul(0,H)+1$, $\mul(-1,G)=\mul(-1,H)$; or
  \item[\rm(ii)] $\mul(0,G)=\mul(0,H)$, $\mul(-1,G)=\mul(-1,H)+1$.
\end{itemize}
\end{lem}
\begin{proof}{Let $G=\nsg(m_1,\ldots,m_h;n_1,\ldots,n_h)$.
First assume that $n_1=\cdots=n_h=m_1=\cdots=m_{h-1}=1$. As $G$ is not an anti-regular graph, we have $m_h\ge3$.
Let $v\in U_h$ and $H=G-v$. Then $H=\nsg(m_1,\ldots,m_{h-1},m_h-1;n_1,\ldots,n_h)$. So by Lemma~\ref{threshmult},
$$\mul(0,G)=2=\mul(0,H)+1,~\mul(-1,G)=0=\mul(-1,H),$$
and so we are done. So we may assume that $n_k\ge2$ for some $1\le k\le h$ or $m_j\ge2$ for some $1\le j\le h-1$.
If $n_k\ge2$, let $v\in V_k$ and $H=G-v$. Then
$$H=\nsg(m_1,\ldots,m_h;n_1,\ldots,n_{k-1},n_k-1,n_{k+1},\ldots,n_h).$$
So by Lemma~\ref{threshmult},
$$\mul(0,G)=\mul(0,H),~\mul(-1,G)=\mul(-1,H)+1.$$
If $m_j\ge2$ for some $1\le j\le h-1$, the result follows similarly.
}\end{proof}
By $\eta_+(G)$ we denote the smallest positive eigenvalue of $G$ and by $\eta_-(G)$ we denote the largest eigenvalue of $G$ less than $-1$.
The unique connected anti-regular graph on $n\geq 2$ vertices is denoted by $A_n$.

\begin{lem}\label{eta} {\rm(\cite{alps})} For anti-regular graphs we have
 $$\eta_-(A_n)<(-1-\sqrt2)/2,~~\hbox{and}~~~(-1+\sqrt2)/2<\eta_+(A_n).$$
 \end{lem}

We are now in a position to prove the main result of the paper.

\begin{thm}\label{main}
Other than the trivial eigenvalues $-1,0$, the interval $\left[(-1-\sqrt2)/2,\, (-1+\sqrt2)/2\right]$ does not contain an eigenvalue of any threshold graph.
\end{thm}
\begin{proof}{Let $G$ be a threshold graph of order $n$. We proceed by induction on $n$. With no loss of generality, we may assume that $G$ is connected.
The assertion holds if $n\le3$, so we assume $n\ge4$.
Recall that threshold graphs have no eigenvalues in the interval $(-1,~0)$.
 If $G$ is an anti-regular graph, then we are done by Lemma~\ref{eta}.
 So assume that $G$ is not an anti-regular graph.
 Let $v$ and $H=G-v$ be as given in Lemma~\ref{G-v}.
Let $\la_1\ge\cdots\ge\la_n$ and $\mu_1\ge\cdots\ge\mu_{n-1}$ be the eigenvalues of $G$ and $H$, respectively.
We may suppose that for some $t$,
\begin{equation}\label{mu}
\mu_{t-\ell-1}>\mu_{t-\ell}=\cdots=\mu_{t-1}=0>\mu_t=\cdots=\mu_{t+j-1}=-1>\mu_{t+j}.
\end{equation}
It is possible that either $\ell=0$ (i.e. $H$ has no $0$ eigenvalue) or $j=0$ (i.e. $H$ has no $-1$ eigenvalue) but we have $j+\ell\ge1$.
By the induction hypothesis,
$$\mu_{t-\ell-1}=\eta_+(H)>\frac{-1+\sqrt2}{2}~~\hbox{and}~~\mu_{t+j}=\eta_-(H)<\frac{-1-\sqrt2}{2}.$$
By interlacing, from \eqref{mu} we have
\begin{align*}
\la_{t-\ell}\ge\la_{t-\ell+1}=\cdots=\la_{t-1}=0\ge&\,\la_t\ge\la_{t+1}=\cdots=\la_{t+j-1}=-1\ge\la_{t+j},\\
\la_{t-\ell-1}\ge\mu_{t-\ell-1}&~~\hbox{and}~~\mu_{t+j}\ge\la_{t+j+1}.
\end{align*}
As $G$ is not an anti-regular graph, the case (i) or (ii) of Lemma~\ref{G-v} occurs.
If the case (i) occurs, then   $\mul(0,G)=\ell+1=\mul(0,H)+1$, $\mul(-1,G)=j=\mul(-1,H)$. This is only possible if
$\la_{t-\ell}=\la_t=0$ and $\la_{t+j}=-1$ if $j\ge1$.
If the case (ii) occurs, then   $\mul(0,G)=\ell=\mul(0,H)$, $\mul(-1,G)=j+1=\mul(-1,H)+1$ which implies that
 $\la_t=\la_{t+j}=-1$ and $\la_{t-\ell}=0$ if $\ell\ge1$. So in any case, $\la_{t-\ell},\ldots,\la_{t+j}\in\{-1,0\}$.
It turns out that
\begin{align*}
\eta_+(G)&=\la_{t-\ell-1}\ge\mu_{t-\ell-1}=\eta_+(H)>\frac{-1+\sqrt2}{2},\\
\eta_-(G)&=\la_{t+j+1}\le\mu_{t+j}=\eta_-(H)<\frac{-1-\sqrt2}{2}.
\end{align*}
The result now follows.
}\end{proof}

\section*{Acknowledgments}
I would like to thank Cesar Aguilar for pointing out an error in an earlier version of the paper.
The research of the author was in part supported by a grant from IPM (No. 98050211).

\end{document}